\documentclass[oneside, 11pt, a4paper]{amsart}
\usepackage[utf8]{inputenc}
\usepackage{amsmath,amstext,amsopn,amssymb, amsfonts, amsthm, mathtools}
\usepackage[dvipsnames]{xcolor}
\usepackage{geometry}
\usepackage{bbm}
\usepackage{float}
\usepackage{pdfsync}
\usepackage{tikz}
\usepackage{enumitem}
\usepackage{subfig}
\usepackage{caption}
\usepackage{graphicx}
\usepackage{mathrsfs}
\usepackage[ruled,lined,procnumbered,linesnumbered]{algorithm2e}
\DeclareMathAlphabet{\mathpzc}{OT1}{pzc}{m}{it}

\usepackage[pdftex,bookmarks,colorlinks,breaklinks]{hyperref}  
\definecolor{dullmagenta}{rgb}{0.4,0,0.4}   
\definecolor{darkblue}{rgb}{0,0,0.4}
\definecolor{darkgreen}{rgb}{0,0.4,0}
\hypersetup{linkcolor=darkblue,citecolor=blue,filecolor=dullmagenta,urlcolor=darkblue} 

\newtheorem{TheoremLetter}{Theorem}
{}

\newtheorem{definition}{Definition}
\newtheorem*{definition*}{Definition}

\newtheorem{theorem}{Theorem}
\newtheorem*{theorem*}{Theorem}

\newtheorem*{conjecture*}{Conjecture}

\newtheorem*{question*}{Question}

\newtheorem{lemma}[theorem]{Lemma}
\newtheorem*{lemma*}{Lemma}

\newtheorem{corollary}[theorem]{Corollary}
\newtheorem*{corollary*}{Corollary}

\newtheorem{remark}[theorem]{Remark}
\newtheorem*{remark*}{Remark}
\numberwithin{equation}{section}
\numberwithin{theorem}{section}

\makeatletter
\newcommand{\customlabel}[2]{%
   \protected@write \@auxout {}{\string \newlabel {#1}{{#2}{\thepage}{#2}{#1}{}} }%
   \hypertarget{#1}{#2}
}
\makeatother

\def\XXint#1#2#3{{\setbox0=\hbox{$#1{#2#3}{\int}$}
     \vcenter{\hbox{$#2#3$}}\kern-.5\wd0}}

\newcommand{\supp}{\operatorname{supp}}

\newcommand{\sh}{\operatorname{sh}}

\newcommand{\ApproxCarleson}{\operatorname{ApproximateCarleson}}
\newcommand{\car}{\operatorname{Carleson}}

\newcommand{\sparse}{\operatorname{Sparse}}

\begin{document}
\thanks{Supported by Grant MICIN/AEI/PID2019-105599GB-I00}
\address{Universidad Aut\'onoma de Madrid}
\author{Guillermo Rey}
\email{guillermo.rey@uam.es}
\title{Greedy approximation algorithms for sparse collections}
\begin{abstract}
  We describe a greedy algorithm that approximates the Carleson constant of a collection of general sets.
  The approximation has a logarithmic loss in a general setting, but is optimal up to a constant with only mild geometric assumptions.
  The constructive nature of the algorithm gives additional information about the almost-disjoint structure of sparse collections.

  As applications, we give three results for collections of axis-parallel rectangles in every dimension.
  The first is a constructive proof of the equivalence between Carleson and sparse collections, first shown by H\"anninen.
  The second is a structure theorem proving that every collection $\mathcal{E}$ can be partitioned into $\mathcal{O}(N)$ sparse subfamilies
  where $N$ is the Carleson constant of $\mathcal{E}$.
  We also give examples showing that such a decomposition is impossible when the geometric assumptions are dropped.
  The third application is a characterization of the Carleson constant involving only $L^{1,\infty}$ estimates.
\end{abstract}
\maketitle

\section{Introduction}
\label{section:intro}

Consider a collection $\mathcal{E}$ of measurable sets in $\mathbb{R}^d$.
We say that $\mathcal{E}$ is $\eta$-sparse if for every
$R$ in $\mathcal{E}$ there exists a subset $E(R) \subseteq R$ satisfying $|E(R)| \geq \eta |R|$ such that the family $\{E(R)\}$ is pairwise-disjoint.
The number $\eta$ quantifies how much overlap exists in $\mathcal{E}$ in a scale-invariant way. In particular, the closer $\eta$ is to $1$ the closer $\mathcal{E}$ is to being pairwise-disjoint.
A closely related quantity is the Carleson constant of $\mathcal{E}$.
For any collection $\mathcal{F}$ let $\sh(\mathcal{F}) = \bigcup_{R \in \mathcal{F}} R$ be its \emph{shadow},
then we say that $\mathcal{E}$ satisfies the Carleson condition with constant $C$ if
\begin{align*}
  \sum_{R \in \mathcal{F}} |R| \leq C |\sh(\mathcal{F})|
\end{align*}
for all subcollections $\mathcal{F} \subseteq \mathcal{E}$.
The best constant in the inequality above is usually called the Carleson constant of $\mathcal{E}$.

These notions have been used
extensively in harmonic analysis, for example in connection
with the boundedness of maximal functions (cf. \cite{CordobaZygmundsConjecture}, \cite{cf}). In recent years they have also
gained a lot of attention for their applications to weighted
inequalities, we direct the interested reader to \cite{pereyra_sparse} for a nice review in this
direction.

It is very easy to see that $\eta$-sparse collections satisfy the Carleson condition with constant $\eta^{-1}$.
With more work one can show that the converse is also true when $\mathcal{E}$ consists of dyadic intervals (or squares, cubes, etc.), see Lemma 6.3 in \cite{LNBOOK}.
This can be done exploiting the strong nestedness property of dyadic intervals, and in fact this structural property allows one to explicitly find the sets $E(R)$ in the definition above.
In particular, in the (one-parameter) dyadic setting the Carleson condition becomes \emph{local}, being equivalent to
\begin{align*}
  \sum_{\substack{S \in \mathcal{E}\\ S \subseteq R}} |S| \leq C |R|
\end{align*}
for all $R$ in $\mathcal{E}$.

This locality is sadly lacking in general, failing even in the two-parameter setting where, instead of dyadic intervals, one works with collections consisting of axis-parallel dyadic rectangles.
This was shown in \cite{carleson1974counter} with what is now known as Carleson's counterexample, see \cite{tao_carleson_counter} or \cite{muscalu_schlag}.

The equivalence between Carleson and sparse collections was first shown by T. H\"anninen in \cite{Hanninen}, who adapted some ideas of L. Dor from \cite{dor} to prove the existence of the
sets $E(R)$ in the definition. The proof, which at its core uses a convexity argument together with Hahn-Banach's separation theorem, is strikingly clean but gives no clue
about \emph{how} the sets $\{E(R)\}$ can be found or about their structure. A more geometric proof was later found in \cite{BarronThesis}, but this proof is also non-constructive.

The main purpose of this article is to describe a greedy algorithm that is able to construct the sets $\{E(R)\}$ for any Carleson collection $\mathcal{E}$.
With no geometric assumptions on $\mathcal{E}$ the algorithm has a logarithmic loss,
but if one imposes some geometric structure then the algorithm provides sets that are
optimal up to an absolute constant.
The constructive nature of our methods allows us to prove a structural theorem about sparse collections with only mild geometric assumptions,
for example valid for axis-parallel rectangles in every dimension.

Before stating the main results, let us begin with some definitions.
We will frequently use the words \emph{collection} and \emph{family} to mean an unordered sequence,
instead of the usual definition of \emph{set}. In particular we allow repeated elements.
For simplicity, all of our collections will be assumed finite.

We can define the Carleson constant of $\mathcal{E}$ with respect to a measure $\mu$ as
\begin{align}\label{def:car}
  \|\mathcal{E}\|_{\car(\mu)} = \sup \Biggl\{ \frac{1}{\mu(\sh(\mathcal{F}))} \sum_{R \in \mathcal{F}} \mu(R) : \mathcal{F} \subseteq \mathcal{E}\Biggr\}.
\end{align}
We will write just $\|\mathcal{E}\|_{\car}$ when $\mu$ is the Lebesgue measure.

For general measures, the straightforward generalization of sparse collection is not equivalent to the Carleson condition above (one needs $\mu$ to have no point masses).
This can be readily seen with the example
\begin{align*}
  \mathcal{E} = \{\underbrace{1, \dots, 1}_{N\text{ times}}\} \qquad \mu = \text{Counting measure}.
\end{align*}

Instead, we can extend it as follows
\begin{definition} \label{def:sparse}
  We say that $\mathcal{E}$ is
  $\eta$-sparse with respect to the measure $\mu$ if one can find non-negative functions $\varphi_R \geq 0$ for each $R$ in $\mathcal{E}$ such that
  \begin{align}
    \int_R \varphi_R &\geq \eta \mu(R) \label{def:sparse:1}\\
    \sum_R \varphi_R &\leq 1. \label{def:sparse:2}
  \end{align}
  We call the best constant $\eta$ above the sparse constant of $\mathcal{E}$ with respect to $\mu$, that is:
  \begin{align*}
    \|\mathcal{E}\|_{\sparse(\mu)} = \sup\{\eta \geq 0:\, \text{$\mathcal{E}$ is $\eta$-sparse with respect to $\mu$}\}.
  \end{align*}
\end{definition}
This is only a slight generalization of the previous definition (one can just take $\varphi_R = \mathbbm{1}_{E(R)}$ to recover the original).
In fact, when the measure $\mu$ has no point masses, one can use a convexity argument like Lemma 2.3 from \cite{dor} to show that the definitions are equivalent.
With this notation the Carleson-sparse equivalence becomes
\begin{align*}
  \|\mathcal{E}\|_{\car(\mu)} \|\mathcal{E}\|_{\sparse(\mu)} = 1.
\end{align*}

We are now ready to describe our algorithm and main results.

Suppose one were to compute $\|\mathcal{E}\|_{\car}$ directly with \eqref{def:car}.
The definition involves computing a certain sum for each of the subcollections $\mathcal{F} \subseteq \mathcal{E}$,
so the process quickly becomes intractable as the cardinality of $\mathcal{E}$ grows.

One could instead try to find functions $\varphi_R$ as in Definition \ref{def:sparse}.
However, this approach quickly runs into problems since the two conditions \eqref{def:sparse:1} and \eqref{def:sparse:2} are in direct opposition.
Namely, \eqref{def:sparse:1} requires that the average of each $\varphi_R$ be at least $\eta$, which together with \eqref{def:sparse:2} means
that each $\varphi_R$ cannot be much smaller than $\eta$ on a large portion of $R$. But also, \eqref{def:sparse:2} implies that the functions
cannot all be larger than $\eta$ at the same place.

If $\mathcal{E}$ consists of only two elements $\{R_1, R_2\}$ then the problem becomes very easy: one can just set $\varphi_{R_i}$ to be
$1$ on the symmetric difference of $R_1$ and $R_2$ and then equitably distribute the mass in $R_1 \cap R_2$ among the two in proportion to their masses.
This could lead to an induction algorithm, but one easily sees that, even with simple examples, earlier choices of the functions $\varphi_R$ can
make the choice of the $(n+1)$-th function impossible, especially when the earlier choices do not take the global situation into account.

This suggests that we choose $\varphi_R$ in a way that guarantees, independently of the choice of $\varphi_S$ for $S \neq R$, that the sum of all
the functions remains bounded by $1$. In particular, we would like to find $R$ so that arbitrarily solving the sub-problem for $\mathcal{E}\setminus R$
still leaves space to choose $\varphi_R$ appropriately.

We are not able to do this in general without at least some geometric information about $\mathcal{E}$. However, we \emph{can} find $R$ so that, if we
choose the following functions in a special way (independent of the choice of $\varphi_R$), then there always exists a
choice of $\varphi_R$ that is valid up to a logarithmic factor. In particular, this strategy leads to functions $\varphi_R$ satisfying
\begin{align*}
  \int_R \varphi_R \,d\mu \gtrsim \frac{\eta}{\log\eta^{-1}}\mu(R),
\end{align*}
where $\eta = \|\mathcal{E}\|_{\sparse(\mu)}$.

In order to remove the logarithmic loss we can use the maximal operator associated to $\mathcal{E}$:
\begin{align*}
  \mathcal{M}_{\mathcal{E}}^{\mu} f = \sup_{R \in \mathcal{E}} \mathbbm{1}_R \frac{1}{\mu(R)} \int_R |f|\, d\mu.
\end{align*}
The geometric condition alluded to previously is related to the restricted weak-type boundedness of $\mathcal{M}_{\mathcal{E}}^\mu$.
In particular, if there exists an $\eta < 1$ such that
\begin{align}\label{rwt}
  \mu\bigl(\bigl\{ x:\, \mathcal{M}_{\mathcal{E}}^\mu(\mathbbm{1}_E)(x) > \eta \bigr\}\bigr) \leq M \mu(E)
\end{align}
for all sets $E$, then the strategy described above leads to an algorithm that finds functions $\varphi_R$ satisfying
\begin{align*}
  \int_R \varphi_R \,d\mu \gtrsim \|\mathcal{E}\|_{\sparse(\mu)} \mu(R),
\end{align*}
where the implied constant depends only on $M$.

The inequality \eqref{rwt} would follow from the the restricted weak-type $(1,1)$ of $\mathcal{M}_{\mathcal{E}}$, so for example it holds
for the Lebesgue measure when $\mathcal{E}$ consists of axis-parallel rectangles, cubes, balls, etc.
One can also consider other measures,
for example in the one-parameter dyadic case $\mathcal{M}^\mu$ is weak-type $(1,1)$-bounded for any measure $\mu$, so
\eqref{rwt} is true whenever $\mathcal{E}$ consists of dyadic intervals (or squares, cubes, etc.)
In two or more parameters the weak-type fails for general measures, but does hold when $d\mu(x) = w(x)dx$ and $w$ is a strong $A_\infty$ weight,
as shown by R. Fefferman in \cite{fefferman_strong_diff}.

Our algorithm finds the functions $\varphi_R$ in Definition \ref{def:sparse} assuming only that \eqref{rwt} holds,
so it immediately gives a constructive proof of the equivalence between the sparse and Carleson conditions.
Without \eqref{rwt} we can constructively prove the equivalence, but only up to a logarithmic factor.

As another application of the algorithm we can prove the following structural property of sparse collections
\begin{TheoremLetter} \label{SplitTheorem}
  Let $\mathcal{E}$ be a collection of sets and suppose $\eqref{rwt}$ holds.
  Then there exists a partition into $\mathcal{O}(\|\mathcal{E}\|)$ subcollections $\{\mathcal{E}_i\}$
  satisfying
  \begin{align*}
    \|\mathcal{E}_i\|_{\car(\mu)} \lesssim 1.
  \end{align*}
\end{TheoremLetter}
This result is proven as a special case of Theorem \ref{split:theorem}, which is a more precise version where we track all the constants.
As an application of Theorem \ref{split:theorem} we can draw a connection with the notion of $(P_1)$ sequence introduced in \cite{cf}.
In particular, we can show that, under the same geometric hypothesis of Theorem \ref{SplitTheorem}, every Carleson collection
can be split into a finite number of $(P_1)$ sequences. We will defer the definition of $(P_1)$ sequence until section \ref{section:split}
where this connection is explained in Remark \ref{p1remark}.

We also show that there are situations where such an splitting is impossible in the absence of an estimate on $\mathcal{M}_{\mathcal{E}}$.
In particular we have, for the Lebesgue measure on $\mathbb{R}$:
\begin{TheoremLetter} \label{NoSplitTheorem}
  For every $\Lambda \geq 2$ and every integer $N \geq 1$ there exists a collection $\mathcal{E}$ of subsets of $\mathbb{R}$ with $\|\mathcal{E}\|_{\car} \leq \Lambda$
  such that for any partition
  \begin{align*}
    \mathcal{E} = \mathcal{E}_1 \cup \dots \cup \mathcal{E}_N
  \end{align*}
  there exists at least one $i \in \{1, \dots, N\}$ for which $\|\mathcal{E}_i\|_{\car} \gtrsim \Lambda$.
\end{TheoremLetter}
Even if we restrict $\mathcal{E}$ to consist of only dyadic rectangles we can produce an example similar to this last one, but for a specially designed measure $\mu$
for which \eqref{rwt} does not hold.

Another application of the algorithms described in this article is that we can weaken the definition of the Carleson condition
to require only weak-type instead of strong $L^1$ estimates.
\begin{TheoremLetter} \label{weaktype11}
  Suppose \eqref{rwt} holds and let $M < \infty$ be the constant in the inequality.
  Then
  \begin{align*}
    \|\mathcal{E}\|_{\car(\mu)} \lesssim \sup\Bigl\{ \frac{1}{\mu(\mathcal{F})} \Bigl\|\sum_{R \in \mathcal{F}} \mathbbm{1}_R\Bigr\|_{L^{1,\infty}(\mu)} :\, \mathcal{F} \subseteq \mathcal{E}\Bigr\},
  \end{align*}
  where the implied constant depends only on $M$.
\end{TheoremLetter}

In section \ref{section:general} we describe a general algorithm to approximate the Carleson constant of a collection following the strategy described here.

In section \ref{section:maximal} we show how to modify the algorithm from the previous section to remove the logarithmic loss, conditional on inequality \eqref{rwt}.
Theorem \ref{weaktype11} is proven at the end of this section.

Finally, in section \ref{section:split} we prove Theorem \ref{SplitTheorem} by inductively constructing said partition, and prove Theorem \ref{NoSplitTheorem} as well as
the dyadic version with explicit examples.

\section{An algorithm for general collections}
\label{section:general}

In this section $\mu$ will always denote a fixed positive measure.
All sets will also be assumed to be of finite $\mu$-measure.
For any collection $\mathcal{E}$ of sets define its \emph{height function}
\begin{align*}
  h_{\mathcal{E}} = \sum_{R \in \mathcal{E}} \mathbbm{1}_R.
\end{align*}
Carleson's condition asserts a uniform bound on the average height of all subcollections. Indeed, if we denote the average height by
\begin{align*}
  \Lambda_\mu(\mathcal{E}) = \frac{1}{\mu(\sh(\mathcal{F}))}\int h_{\mathcal{E}} \, d\mu,
\end{align*}
then Carleson's condition becomes $\|\mathcal{E}\|_{\car(\mu)} = \sup\bigl\{ \Lambda_\mu(\mathcal{F}):\, \mathcal{F} \subseteq \mathcal{E}\bigr\}$.

The next lemma is the main iteration step in our algorithm.
\begin{lemma} \label{general:selection_lemma}
  Let $\mathcal{E}$ be a collection of sets and suppose $\Lambda := \Lambda_\mu(\mathcal{E}) < \infty$.
  Define\footnote{Here and throughout we will take the convention
  that $\frac{0}{0} = 0$.} for every $R$ in $\mathcal{E}$
  \begin{align*}
    g_R &:= \frac{\mathbbm{1}_R}{h_{\mathcal{E}}} \mathbbm{1}_{\{x:\, h_{\mathcal{E}}(x) \leq 2\Lambda\}},
  \end{align*}

  Then there exists at least one $R$ in $\mathcal{E}$ such that
  \begin{align} \label{general:selection_lemma:claim}
    \int_R g_R \, d\mu \geq \frac{1}{2\Lambda}|R|.
  \end{align}
\end{lemma}
\begin{proof}
  By definition we have
  \begin{align} \label{general:selection_lemma:cc}
    \frac{1}{\Lambda}\int h_{\mathcal{E}} \, d\mu = \mu(\sh(\mathcal{E})).
  \end{align}

  Define the set $G = \{x \in \sh(\mathcal{E}):\, h_{\mathcal{E}}(x) \leq 2\Lambda\}$.
  If $\mu(G) = \mu(\sh(\mathcal{E}))$ we are done, so we can assume that $\mu(G) < \mu(\sh(\mathcal{E}))$.
  Then, from Markov's inequality and \eqref{general:selection_lemma:cc} we can estimate
  \begin{align*}
    \mu(G) &= \mu(\sh(\mathcal{E})) - \mu(\{x:\, h_{\mathcal{E}} > 2\Lambda\}) \\
           &> \mu(\sh(\mathcal{E})) - \frac{1}{2\Lambda}\int h_{\mathcal{E}} \, d\mu \\
           &\geq \frac{1}{2}\mu(\sh(\mathcal{E})),
  \end{align*}
  and hence $\mu(G) > \frac{1}{2}\mu(\sh(\mathcal{E}))$.

  Suppose by way of contradiction that \eqref{general:selection_lemma:claim} fails for all $R$. That, is, for all $R$ in $\mathcal{E}$
  \begin{align*}
    \int_R g_R \, d\mu < \frac{1}{2\Lambda}|R|.
  \end{align*}

  Then
  \begin{align*}
    \mu(G) = \int \mathbbm{1}_G \, d\mu &= \int \sum_{R \in \mathcal{E}} g_R \, d\mu \\
           &< \sum_{R \in \mathcal{E}} \frac{1}{2\Lambda}\mu(R)
           = \frac{1}{2\Lambda} \int h_{\mathcal{E}} \, d\mu \leq \frac{1}{2}\mu(\sh(\mathcal{E})).
  \end{align*}
  This means $\mu(G) < \frac{1}{2}\mu(\sh(\mathcal{E}))$, which is a contradiction.
\end{proof}

If we iterate this lemma we obtain the algorithm described in the introduction.

\begin{algorithm}[H]
  \DontPrintSemicolon
  \caption{$\ApproxCarleson(\mathcal{E})$} \label{general:algo}
  \Begin{
    Set $A = 1$.\;
  }
  \While{
    $\mathcal{E} \neq \emptyset$
  }{
    Set $A = \max(A, \Lambda_\mu(\mathcal{E}))$.\;
    Construct the functions $g_R$ as in Lemma \ref{general:selection_lemma}.\;
    \For{$R \in \mathcal{E}$}{
      \If{$\int g_R \, d\mu\geq \frac{\mu(R)}{2\Lambda_\mu(\mathcal{E})}$}{
        Remove $R$ from $\mathcal{E}$.\;
        Assign $f_R := g_R$ and $\Lambda_R := \Lambda_\mu(\mathcal{E})$.\;
        Go to line 5.\;
      }
    }
  }
  \KwResult{
    The constant $A$, the sequence $\{\Lambda_R\}$, and the functions $\{f_R\}$\;
  }
\end{algorithm}

The purpose of lines 7 through 13 is to find $R$ so that
\begin{align*}
  \int g_R\, d\mu \geq \frac{\mu(R)}{2\Lambda_{\mu}(\mathcal{E})}.
\end{align*}
Lemma \ref{general:selection_lemma} shows that such an $R$ always exists, so the algorithm removes one element from $\mathcal{E}$ at a time and thus always terminates.
The following theorem shows that the approximation of $\|\mathcal{E}\|_{\car(\mu)}$, namely $A$,
is correct up to a logarithm.
\begin{theorem} \label{general:correctness}
  Let $A$ be the constant obtained as the result of running the algorithm on a collection $\mathcal{E}$.
  Then we have
  \begin{align*}
     A \leq \|\mathcal{E}\|_{\car(\mu)} \lesssim A \log(e+A).
  \end{align*}
\end{theorem}
\begin{proof}
  The inequality $A \leq \|\mathcal{E}\|_{\car(\mu)}$ is trivial since $A$ is always one of the possible elements in the supremum
  of the definition of $\|\mathcal{E}\|_{\car(\mu)}$.
  Suppose we could show
  \begin{align} \label{general:theorem:good_bound}
    \sum_{R \in \mathcal{E}} f_R &\leq C\log(e+A).
  \end{align}
  Then we could set for each $R$ in $\mathcal{E}$
  \begin{align*}
    \varphi_R := \frac{f_R}{C \log(e+A)}.
  \end{align*}
  Now these functions obviously satisfy $\sum_{R \in \mathcal{E}} \varphi_R \leq 1$ and
  \begin{align*}
    \int_R \varphi_R\, d\mu &= \frac{1}{C \log(e+A)} \int f_R \, d\mu \\
                     &\geq \frac{1}{C \log(e+A)} \frac{1}{2A}
  \end{align*}
  for all $R$ in $\mathcal{E}$. According to Definition \ref{def:sparse}, this would make $\mathcal{E}$ an $\eta$-sparse collection with
  \begin{align*}
    \eta = \frac{1}{2CA\log(e+A)}
  \end{align*}
  and hence $\|\mathcal{E}\|_{\car(\mu)} \leq 2CA\log(e+A)$.

  We now proceed to prove \eqref{general:theorem:good_bound}.
  For any two $R, S \in \mathcal{E}$ set $R \prec S$ if and only if $R$ was removed from $\mathcal{E}$ before $S$ (in line 9). Define
  $\mathcal{E}_{\prec R} = \{S \in \mathcal{E}:\, S \prec R\}$ and $\mathcal{E}_{\succeq R} = \{S \in \mathcal{E}:\, S \succeq R\}$.
  Since $\prec$ is a total order, we have that $\mathcal{E} = \mathcal{E}_{\prec R} \sqcup \mathcal{E}_{\succeq R}$ for each $R$ in $\mathcal{E}$.
  Set
  \begin{align*}
    \mathcal{B}(x) = \{R \in \mathcal{E}:\, x \in \supp f_R\}
  \end{align*}
  and let $(R_1, R_2, R_3, \dots)$ be the elements of $\mathcal{B}$ sorted in increasing order by $\prec$.
  Observe that the cardinatlity $N$ of $\mathcal{B}(x)$ satisfies on the one hand
  \begin{align*}
    N \leq h_{\mathcal{E}_{\succeq R_1}}(x).
  \end{align*}
  On the other hand, if $x$ is in $\supp f_{R_1}$, then $h_{\mathcal{E}_{\succeq R_1}}(x) \leq 2\Lambda_{R_1}$ and thus
  $N \leq 2\Lambda_{R_1} \leq 2A$. So
  \begin{align*}
    \sum_{R \in \mathcal{E}} f_R(x) = \sum_{n=1}^N f_{R_n}.
  \end{align*}
  Note that, by construction we have $f_{R_n}(x) \leq \frac{1}{1+N-n}$, therefore
  \begin{align*}
    \sum_{R \in \mathcal{E}} f_R(x) &\leq \sum_{n=1}^N \frac{1}{1+N-n} \\
                                    &\lesssim \log(e+A)
  \end{align*}
  and we are done.
\end{proof}

\section{An improvement with the maximal function}
\label{section:maximal}

Recall the maximal operator associated to the family $\mathcal{E}$ and the measure $\mu$ from the introduction:
\begin{align*}
  \mathcal{M}_{\mathcal{E}}^\mu f = \sup_{R \in \mathcal{E}}  \frac{\mathbbm{1}_R}{\mu(R)}\int_R |f|\, d\mu.
\end{align*}
The measure $\mu$ will be fixed throughout this section, so we will abbreviate $\mathcal{M}_{\mathcal{E}} f := \mathcal{M}_{\mathcal{E}}^\mu f$.

We will show how Algorithm \ref{general:algo} can be slightly modified to give an essentially-optimal approximation of $\|\mathcal{E}\|_{\car}$ whenever
$\mathcal{M}$ satisfies the condition in \eqref{rwt} which, we recall, was that for a fixed number $0 < \eta < 1$
\begin{align*}
  \mu\bigl(\bigl\{ x:\, \mathcal{M}_{\mathcal{E}}(\mathbbm{1}_E)(x) > \eta \bigr\}\bigr) \leq M \mu(E)
\end{align*}
uniformly over all measurable sets $E$. We will denote by $M_\eta(\mathcal{E})$ the best constant in this inequality (again, dropping the dependence on $\mu$ for simplicity).

In the proof of Theorem \ref{general:correctness} we showed how the logarithmic loss appears with Algorithm \ref{general:algo}.
In particular, dividing by $h_{\mathcal{E}}$ was needed in order to get a reasonably large value of $\int g_R$,
which is where we the logarithm appears as we end up having to sum the harmonic series.
Dividing by a larger function
would make the integral too small, while a smaller one makes bounding $\sum_R g_R$ harder.

Here we take a different approach. The idea is that, if $M_{\eta}(\mathcal{E})$ is finite, there must be a set $R$ in $\mathcal{E}$
that intersects the high level-set of $h_{\mathcal{E}}$ in only a small portion relative to itself.
The next lemma is the main iteration step of the improved algorithm and is in the same spirit as Lemma \ref{general:selection_lemma}.
\begin{lemma} \label{maximal:selection_lemma}
  Suppose that $M_\eta(\mathcal{E}) < \infty$ and
  \begin{align} \label{maximal:selection_lemma:wt}
    \sup_{\lambda >0} \lambda \mu(\{x:\, h_{\mathcal{E}}(x) > \lambda\}) \leq \Lambda\mu(\sh(\mathcal{E})).
  \end{align}
  Then there must exist at least one $R$ in $\mathcal{E}$ satisfying
  \begin{align} \label{maximal:selection_lemma:claim}
    \mu(\{x \in R:\, h_{\mathcal{E}}(x) \leq 2 \Lambda M_\eta(\mathcal{E})\}) \geq (1-\eta) \mu(R).
  \end{align}
\end{lemma}
\begin{proof}
  To simplify the notation we will abbreviate $M := M_\eta(\mathcal{E})$.
  Suppose \eqref{maximal:selection_lemma:claim} does not hold for any $R$, that is: for every $R$ in $\mathcal{E}$
  \begin{align*}
    \mu(\{x \in R:\, h_{\mathcal{E}}(x) \leq 2M \Lambda\}) < (1-\eta) \mu(R).
  \end{align*}
  Then
  \begin{align*}
    \mu(\{x \in R:\, h_{\mathcal{E}}(x) > 2M \Lambda\}) &= \mu(R) - \mu(\{x \in R:\, h_{\mathcal{E}}(x) \leq 2M \Lambda\}) > \eta\mu(R).
  \end{align*}
  Set $B = \{x \in \sh(\mathcal{E}):\, h_{\mathcal{E}}(x) > 2M\Lambda\}$. This estimate implies
  \begin{align*}
    \mu(B \cap R) > \eta \mu(R) \implies R \subseteq \{\mathcal{M}_{\mathcal{E}}(\mathbbm{1}_B) > \eta\}.
  \end{align*}
  Since \eqref{maximal:selection_lemma:claim} does not hold for any $R$ we in fact have $\sh(\mathcal{E}) \subseteq \{M_{\mathcal{E}}(\mathbbm{1}_B) > \eta\}$.

  By \eqref{maximal:selection_lemma:wt} we can estimate $\mu(B)$ from above as follows:
  \begin{align*}
    \mu(B) \leq \frac{1}{2M\Lambda} \Lambda \mu(\sh(\mathcal{E})) = \frac{\mu(\sh(\mathcal{E}))}{2M}.
  \end{align*}

  Thus, by the finiteness of $M$:
  \begin{align}
    \mu(\sh(\mathcal{E})) &\leq \mu(\{\mathcal{M}_{\mathcal{E}}(\mathbbm{1}_B) > \eta\})
    \stackrel{\eqref{rwt}}{\leq} M \mu(B)
    \leq  \frac{1}{2}\mu(\sh(\mathcal{E})),
  \end{align}
  which is a contradiction.
\end{proof}
Note that by Markov's inequality
\begin{align} \label{maximal:easy}
  \mu(\{x:\, h_{\mathcal{E}}(x) > \lambda\}) \leq \lambda^{-1} \int h_{\mathcal{E}} \,d\mu = \lambda^{-1} \Lambda_{\mu}(\mathcal{E})\mu(\sh(\mathcal{E})).
\end{align}
So in particular condition \eqref{maximal:selection_lemma:wt} holds with $\Lambda \leq \Lambda_{\mu}(\mathcal{E})$.

This lemma shows that one can find a set $R$ in $\mathcal{E}$ with a large subset in which $R$ is guaranteed to have bounded overlap with all the other
sets in $\mathcal{E}$. In particular, if we set
\begin{align} \label{maximal:gf}
  F(R) = \{x \in R:\, h_{\mathcal{E}}(x) \leq 2M_\eta(\mathcal{E})\Lambda^{1,\infty}_\mu(\mathcal{E})\},
\end{align}
where $\Lambda^{1,\infty}_\mu(\mathcal{E})$ is the best constant in \eqref{maximal:selection_lemma:wt}.
Then there must exist at least one $R$ such that $\mu(F(R)) \geq (1-\eta)\mu(R)$.

We can now give the improved version of Algorithm \ref{general:algo}:

\begin{algorithm}[H]
  \DontPrintSemicolon
  \caption{$\ApproxCarleson(\mathcal{E})$ - improved} \label{maximal:algo}
  \Begin{
    Set $A = 1$.
  }
  \While{
    $\mathcal{E} \neq \emptyset$
  }{
    Set $A = \max(A, \Lambda_\mu^{1,\infty}(\mathcal{E}))$.\;
    \For{$R \in \mathcal{E}$}{
      \If{$\mu(F(R)) \geq (1-\eta)\mu(R)$}{
        Remove $R$ from $\mathcal{E}$.\;
        Assign $E(R) := F(R)$ and $\Lambda_R := \Lambda_\mu^{1,\infty}(\mathcal{E})$.\;
        Go to line 5.\;
      }
    }
  }
  \KwResult{
    The constant $A$, the sequence $\{\Lambda_R\}$, and the sets $\{E(R)\}$\;
  }
\end{algorithm}

As in the proof of Theorem \ref{general:correctness}, the order in which elements are removed from $\mathcal{E}$ is important.
Set $R \prec S$ if and only if $R$ was removed before $S$ by Algorithm \ref{maximal:algo}. Set also
\begin{align*}
  \mathcal{E}_{\succeq R} = \{S \in \mathcal{E}:\, S \succeq R\}
\end{align*}
with the natural definition of $\succeq$ in terms of $\prec$.
The important property given by this order is the following inequality for the level sets of $h_{\mathcal{E}}$:
\begin{align} \label{maximal:key_ineq}
  \mu(\{x \in R:\, h_{\mathcal{E}_{\succeq R}}(x) \leq 2M \Lambda^{1,\infty}_{\mu}(\mathcal{E}_{\succeq R}) \}) \geq (1-\eta)\mu(R),
\end{align}
where we have abbreviated $M = M_{\eta}(\mathcal{E})$. Note that $\Lambda^{1,\infty}_{\mu}(\mathcal{E}_{\succeq R}) \leq A \leq \|\mathcal{E}\|_{\car(\mu)}$.

The next theorem shows that estimates like these imply upper bounds on the Carleson constant of $\mathcal{E}$.
\begin{theorem} \label{maximal:general_thm}
  Let $\mathcal{E}$ be a collection totally ordered by some binary relation $\prec$. Suppose that we have
  \begin{align} \label{maximal:upperbound:sufficient}
    \mu(\{x \in R:\, h_{\mathcal{E}_{\succeq R}}(x) \leq \Lambda \}) \geq \eta \mu(R)
  \end{align}
  for all $R$ in $\mathcal{E}$. Then $\|\mathcal{E}\|_{\car(\mu)} \leq \Lambda \eta^{-1}$.
\end{theorem}
\begin{proof}
  We will show that $\|\mathcal{E}\|_{\sparse(\mu)} \geq \eta\Lambda^{-1}$.

  For each $R$ in $\mathcal{E}$ let $E(R) = \{x \in R:\, h_{\mathcal{E}_{\succeq R}}(x) \leq \Lambda\}$ and define
  the functions
  \begin{align*}
    \varphi_R = \frac{\mathbbm{1}_{E(R)}}{\Lambda}.
  \end{align*}
  By \eqref{maximal:upperbound:sufficient} we have
  \begin{align*}
    \int_R \varphi_R \,d\mu \geq \frac{\eta}{\Lambda}\mu(R).
  \end{align*}

  For any point $x$ in $\sh(\mathcal{E})$ let $\mathcal{B}(x) = \{S \in \mathcal{E}:\, x \in E(S)\}$, then
  \begin{align*}
    \sum_{R \in \mathcal{E}} \varphi_R = \frac{\#(\mathcal{B}(x))}{\Lambda}.
  \end{align*}
  So it suffices to show that $\mathcal{B}(x)$ has at most $\Lambda$ elements.

  Let $N = \#(\mathcal{B}(x))$, and let $R$ be the minimal element of $\mathcal{B}(x)$ with respect to $\prec$.
  Then obviously $h_{\mathcal{E}_{\succeq R}}(x) \geq N$. And since $x \in E(R)$, we must have $N \leq \Lambda$.

  Thus, the functions $\{\varphi_R\}$ satisfy the conditions of Definition \ref{def:sparse} and we are done.
\end{proof}

\begin{corollary} \label{maximal:coro}
  If $M_\eta < \infty$ and $A$ is the output constant of Algorithm \ref{maximal:algo} then
  \begin{align}\label{maximal:theorem:estimate}
    A \leq \|\mathcal{E}\|_{\car(\mu)} \leq 2(1-\eta)^{-1} M_\eta A.
  \end{align}
\end{corollary}
\begin{proof}
  The lower bound is trivial from the definition of Carleson constant. The upper bound follows by combining \eqref{maximal:key_ineq} with
  Theorem \ref{maximal:general_thm}.
\end{proof}

The fact that we only really needed the weak-type bound in \eqref{maximal:selection_lemma:wt} allows us to prove Theorem \ref{weaktype11}:
\begin{proof}
  Suppose
  \begin{align*}
    \|h_{\mathcal{F}}\|_{L^{1,\infty}(\mu)} \leq C_0\mu(\mathcal{F})
  \end{align*}
  for all $\mathcal{F} \subseteq \mathcal{E}$. This means that at each iteration in the algorithm \eqref{maximal:selection_lemma:wt}
  holds with $\Lambda \leq C_0$, thus the constant $A$ output as a result is at most $C_0$ and the claim follows by Corollary \ref{maximal:coro}
\end{proof}

\section{Breaking up sparse collections}
\label{section:split}

We are now ready to the structure theorem  mentioned in the introduction.
\begin{theorem} \label{split:theorem}
  Let $\mathcal{E}$ be an arbitrary countable collection of sets with finite $\mu$-measure,
  and suppose the maximal operator $\mathcal{M}_{\mathcal{E}}^\mu$ satisfies \eqref{rwt} with constant $M = M_{\eta}$.

  Then for any $0 < \gamma < 1-\eta$ there exists a partition of $\mathcal{E}$ into at most
  \begin{align*}
    1+\frac{2M(1-\eta)}{1-\eta-\gamma}\|\mathcal{E}\|_{\car(\mu)}
  \end{align*}
  subcollections $\{\mathcal{E}_i\}$
  satisfying
  \begin{align*}
    \|\mathcal{E}_i\|_{\sparse(\mu)} \geq \gamma.
  \end{align*}

\end{theorem}
\begin{proof}
  After applying Algorithm \ref{maximal:algo} and reversing the order, one obtains a total order $<$ on $\mathcal{E}$ such that
  \begin{align} \label{split:good_order}
    \mu(\{x \in R:\, h_{\mathcal{E}_{\leq R}}(x) > 2 M \|\mathcal{E}\|_{\car(\mu)}\}) \leq \eta \mu(R)
  \end{align}
  for all $R \in \mathcal{E}$.

  Create $N$ empty \emph{buckets} $\{\mathcal{E}_1, \dots, \mathcal{E}_N\}$, where $N$ is a large integer to be chosen later.
  These buckets will be constructed by iteratively inserting elements from $\mathcal{E}$.

  We start with the smallest (with respect to $<$) element in $\mathcal{E}$, which we can insert into an arbitrary bucket, say $\mathcal{E}_1$.
  Let $R$ be any set in $\mathcal{E}$ and assume that we have
  placed all the previous sets $S < R$ in such a way that for all $i$ and $R \in \mathcal{E}_i$:
  \begin{align} \label{split:condition}
    \mu(S \cap \sh(\mathcal{E}_i)) \leq (1-\gamma)\mu(S).
  \end{align}
  We now show that one can also place $R$ into at least one of the buckets while maintaing \eqref{split:condition} with $R$ instead of $S$. Indeed, suppose \eqref{split:condition}
  fails for all the $N$ buckets. Then, for $A_i = \sh(\mathcal{E}_i) \cap R$, we have
  \begin{align} \label{split:big_subsets}
    \mu(A_i) > (1-\gamma)\mu(R)
  \end{align}
  for all $i \in \{1, \dots, N\}$. Set
  \begin{align} \label{split:alpha}
    \alpha = \frac{1-\eta-\gamma}{1-\eta}
  \end{align}
  and let $U = \{x \in R:\, \sum_{i=1}^N \mathbbm{1}_{A_i} \geq \alpha N\}$.
  We will show that
  \begin{align} \label{split:too_big}
    \mu(U) > \eta\mu(R).
  \end{align}
  This will contradict \eqref{split:good_order} if $N \geq \alpha^{-1}2M\|\mathcal{E}\|_{\car(\mu)}$ since we would have
  \begin{align*}
    \{x \in R:\, h_{\mathcal{E}_{\leq R}}(x) > 2 M \|\mathcal{E}\|_{\car(\mu)}\} \supseteq U.
  \end{align*}

  To estimate $\mu(U)$ it is easier to bound the measure of the complement $V = R \setminus U$. If $x$ is in fewer than $\alpha N$ of the subsets $\{A_i\}$, then
  $x$ is in at least $(1-\alpha)N$ of the subsets $\{R \setminus A_i\}$. Thus
  \begin{align*}
    \mu(V) &= \mu\Biggl( \Biggl\{ x \in R:\, \sum_{i=1}^N \mathbbm{1}_{A_i^c} > (1-\alpha)N \Biggr\} \Biggr) \\
        &\leq \frac{1}{(1-\alpha)N} \int_R  \sum_{i=1}^N \mathbbm{1}_{A_i^c}\, d\mu \\
        &= \frac{1}{(1-\alpha)N} \sum_{i=1}^N \mu(R \setminus A_i).
  \end{align*}
  By \eqref{split:big_subsets} we have $\mu(R \setminus A_i) < \gamma \mu(R)$, so
  \begin{align*}
    \mu(V) &< \frac{1}{(1-\alpha)N}\sum_{i=1}^N \gamma\mu(R) \\
           &= \frac{\gamma}{1-\alpha}\mu(R).
  \end{align*}
  Thus, with our choice of $\alpha$ in \eqref{split:alpha}:
  \begin{align*}
    \mu(U) &= \mu(R) - \mu(V) \\
           &> \mu(R) \Bigl( 1 - \frac{\gamma}{1-\alpha} \Bigr) \\
           &= \mu(R) \Bigl( \frac{1-\alpha-\gamma}{1-\alpha} \Bigr) \\
           &= \eta \mu(R),
  \end{align*}
  which is our contradiction.

  Finally, it remains to chose $N$, but this is easy as the smallest integer $N \geq \alpha^{-1} 2 M \|\mathcal{E}\|_{\car(\mu)}$ will suffice.
\end{proof}

\begin{remark} \label{p1remark}
  We would like to note here that Theorem \ref{split:theorem} proves that every Carleson collection of axis-parallel rectangles (or sets for which the associated
  maximal function satisfies \eqref{rwt}) can be decomposed into finitely many collections of type $(P_1)$ in the nomenclature of \cite{cf}. Recall that a sequence $\{R_1, R_2, \dots\}$ is
  of type $(P_1)$ if for every $n \geq 1$
  \begin{align*}
    |R_{n+1} \setminus (R_1 \cup \dots \cup R_n)| \geq \frac{1}{2} |R_{n+1}|.
  \end{align*}
  Observe that, in the proof of Theorem \ref{split:theorem}, the buckets $\mathcal{E}_i$ satsify \eqref{split:condition} which is exactly the $(P_1)$ condition when $\gamma = \frac{1}{2}$.
\end{remark}

We now show that the structure theorem is not true in general.
\begin{theorem} \label{split:counter1}
  For every $\Lambda \geq 2$ and every integer $N \geq 1$ there exists a collection $\mathcal{E}$ of subsets of $\mathbb{R}$ with $\|\mathcal{E}\|_{\car} \leq \Lambda$
  such that for any partition
  \begin{align*}
    \mathcal{E} = \mathcal{E}_1 \cup \dots \cup \mathcal{E}_N
  \end{align*}
  there exists at least one $i \in \{1, \dots, N\}$ for which $\|\mathcal{E}_i\|_{\car} \geq \frac{1}{2}\Lambda$.
\end{theorem}
\begin{proof}
  Fix a large integer $M$ to be chosen later. For any integer $m \geq 1$ define the sets
  \begin{align*}
    R_m = [0,1) \cup [m, m + (\Lambda-1)^{-1}).
  \end{align*}

  If $\mathcal{F}$ is any non-empty subcollection of $\{R_1, R_2, \dots\}$ then
  \begin{align*}
    |\sh(\mathcal{F})| &= 1 + \#(\mathcal{F})(\Lambda-1)^{-1}.
  \end{align*}
  For each $m \geq 1$ let $E(R_m) = [m, m + (\Lambda-1)^{-1})$, then the collection $\{E(R_m)\}$ is pairwise-disjoint and
  \begin{align*}
    \frac{|E(R_m)|}{|R_m|} = \frac{(\Lambda-1)^{-1}}{1+(\Lambda-1)^{-1}} = \frac{1}{\Lambda}.
  \end{align*}
  These two facts mean that the collection $\{R_0, \dots, R_{M-1}\}$ is $\Lambda^{-1}$-sparse, and hence $\|\mathcal{E}\|_{\car} \leq \Lambda$.

  Now let $\mathcal{E}_1 \cup \dots \cup \mathcal{E}_N$ be any partition of $\mathcal{E}$.
  Since $\#\mathcal{E} = M$, there must exist an $i \in \{1, \dots, N\}$ such that $\# \mathcal{E}_i \geq \frac{M}{N}$.
  For this family we have
  \begin{align*}
    \|\mathcal{E}_i\|_{\car} \geq \Lambda_\mu(\mathcal{E}_i) &= \frac{\#\mathcal{E}_i(1 + (\Lambda-1)^{-1})}{1+\#\mathcal{E}_i (\Lambda-1)^{-1}}.
  \end{align*}
  When $M$ is sufficiently large (depending only on $N$ and $\Lambda$) we have
  \begin{align*}
    \frac{\#\mathcal{E}_i(1 + (\Lambda-1)^{-1})}{1+\#\mathcal{E}_i (\Lambda-1)^{-1}} &\geq \frac{1 + (\Lambda-1)^{-1}}{2(\Lambda-1)^{-1}} \\
    &= \frac{\Lambda}{2},
  \end{align*}
  which is what we wanted.
\end{proof}

One may wonder whether one can improve matters by imposing additional geometry on the sets contained in $\mathcal{E}$.
For example, when $\mathcal{E}$ consists of dyadic rectangles in $\mathbb{R}^d$ then Theorem \ref{split:theorem} applies.
However, if one is allowed to change the measure then we can construct an example that behaves like the one in Theorem \ref{split:counter1}.

The construction, which has essentially the same behavior as that of Theorem \ref{split:counter1}, is similar to one used
by R. Fefferman in \cite{fefferman_strong_diff}.
\begin{theorem}
  There exists a measure $\mu$ on $\mathbb{R}^d$ such that for any integers $N \geq 1$ and $\Lambda \geq 2$ there exists
  a finite collection $\mathcal{E}$ of dyadic rectangles with $\|\mathcal{E}\|_{\car(\mu)} \leq \Lambda$ such that any
  partition into $N$ subfamilies has at least one with Carleson constant $\geq \frac{1}{2}\Lambda$.
\end{theorem}
\begin{proof}
  For integers $m$ and $j$ consider the dyadic rectangles
  \begin{align*}
    R_m^j = [0, 2^m) \times [j, j+2^{-m})
  \end{align*}
  and let $\mathcal{S}^j = \{R_m^j:\, m \geq 0\}$. Define also
  \begin{align*}
    E(R_m^j) = [2^{m-1}, 2^m) \times [j + 2^{-m-1}, j+2^{-m}).
  \end{align*}
  Observe that the sets $\{E(R_m^j)\}$ are pairwise-disjoint.

  Choose any set of points $\{x_m^j\}$ such that $x_m^j \in E(R_m^j)$ for every non-negative $m$ and $j$.
  Then define the measure
  \begin{align*}
    \mu = \sum_{j=0}^\infty \Bigl( \delta_{(0,j)} + \sum_{m=0}^\infty (1+j)^{-1}\delta_{x_m^j} \Bigr).
  \end{align*}

  \begin{figure}[H]
    \begin{tikzpicture}[scale=0.5]
      \draw[thick,->] (0,0) -- (16,0);
      \draw[thick,->] (0,0) -- (0,12);
      \node at  (-0.5, 0) {$0$};
      \node at  (-0.5, 5) {$1$};
      \node at  (-0.5, 10) {$2$};

      \node at  (14, 0.5) {$\hdots$};
      \node at  (14, 5.5) {$\hdots$};
      \node at  (2.5, 11) {$\vdots$};

      \draw[draw=black] (0,0) rectangle (05, 05);
      \draw[draw=black] (0,0) rectangle (07, 04);
      \draw[draw=black] (0,0) rectangle (09, 03);
      \draw[draw=black] (0,0) rectangle (11, 02);
      \draw[draw=black] (0,0) rectangle (13, 01);

      \filldraw (00,00) circle [radius=0.05];

      \filldraw (4.9, 4.9) circle [radius=0.05];
      \node at  (4.5, 4.5) {$x_0^0$};

      \filldraw (6.9, 3.9) circle [radius=0.05];
      \node at  (6.5, 3.5) {$x_1^0$};

      \filldraw (8.9, 2.9) circle [radius=0.05];
      \node at  (8.5, 2.5) {$x_2^0$};

      \filldraw (10.9, 1.9) circle [radius=0.05];
      \node at  (10.5, 1.5) {$x_3^0$};

      \filldraw (12.9, 0.9) circle [radius=0.05];
      \node at  (12.5, 0.5) {$x_4^0$};

      \draw[draw=black] (0,5) rectangle (05, 10);
      \draw[draw=black] (0,5) rectangle (07, 09);
      \draw[draw=black] (0,5) rectangle (09, 08);
      \draw[draw=black] (0,5) rectangle (11, 07);
      \draw[draw=black] (0,5) rectangle (13, 06);

      \filldraw (00, 05) circle [radius=0.05];

      \filldraw (4.9, 9.9) circle [radius=0.05];
      \node at  (4.5, 9.5) {$x_0^1$};

      \filldraw (6.9, 8.9) circle [radius=0.05];
      \node at  (6.5, 8.5) {$x_1^1$};

      \filldraw (8.9, 7.9) circle [radius=0.05];
      \node at  (8.5, 7.5) {$x_2^1$};

      \filldraw (10.9, 6.9) circle [radius=0.05];
      \node at  (10.5, 6.5) {$x_3^1$};

      \filldraw (12.9, 5.9) circle [radius=0.05];
      \node at  (12.5, 5.5) {$x_4^1$};
    \end{tikzpicture}
    \caption{The first few rectangles $\{R_m^j\}$ and points $\{x_m^j\}$ (not to scale).}
    \label{fig:stairs}
  \end{figure}
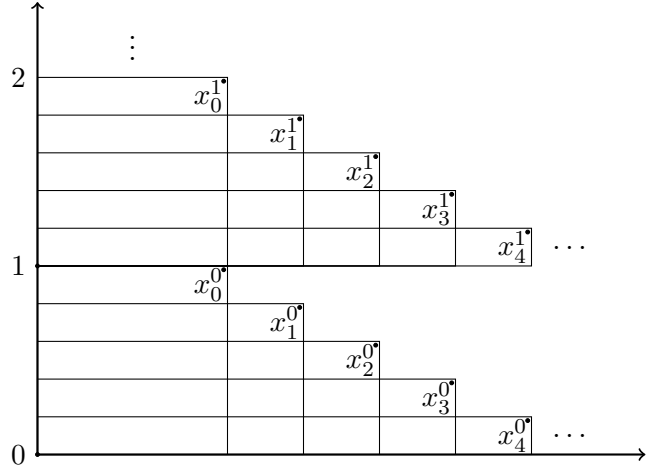

  With this measure we have
  \begin{align*}
    \mu(R_m^j) = 1 + (1+j)^{-1} \quad \text{and} \quad \mu(E(R_m^j)) = (1+j)^{-1},
  \end{align*}
  for all $m$ and $j$.
  As in the proof of Theorem \ref{split:counter1}, for any finite collection $\mathcal{F} \subset \mathcal{S}^j$ we have
  \begin{align}
    \|\mathcal{F}\|_{\car(\mu)} &\leq \frac{1+(1+j)^{-1}}{(1+j)^{-1}} = 2 + j \label{split:counter2:ub} \\
    \|\mathcal{F}\|_{\car(\mu)} &\geq \frac{\#(\mathcal{F})(1+(1+j)^{-1})}{1 + \#(\mathcal{F})(1+j)^{-1}} = \frac{\#(\mathcal{F})(2+j)}{1+j + \#(\mathcal{F})}. \label{split:counter2:lb}
  \end{align}

  Let $M$ be a large integer and take any subset $\mathcal{E}$ from $\mathcal{S}^{\Lambda-2}$ with $\#(\mathcal{E}) = M$.
  By \eqref{split:counter2:ub} we can bound the Carleson constant of $\mathcal{E}$ by $\Lambda$.
  Suppose $\{\mathcal{E}_i\}$ is any partition of $\mathcal{E}$ into $N$ subfamilies. There must exist at least
  one subfamily, say $\mathcal{E}_i$, with $\#(\mathcal{E}_i) \geq M/N$. For this subfamily we have
  by \eqref{split:counter2:lb}
  \begin{align*}
    \|\mathcal{E}_i\|_{\car(\mu)} &\geq \frac{\frac{M}{N}\Lambda}{\Lambda-1 + \frac{M}{N}} \\
    &= \frac{M\Lambda}{N(\Lambda-1) + M}.
  \end{align*}
  The claim follows by taking $M$ so large that
  \begin{align*}
    \frac{M\Lambda}{N(\Lambda-1) + M} \geq \frac{\Lambda}{2}.
  \end{align*}
\end{proof}

\bibliography{bibliography}
\bibliographystyle{abbrv}

\end{document}